\documentclass[12pt]{amsart}
\usepackage{hyperref,graphicx,amssymb,amsthm,amsmath, tikz-cd, array, verbatim, indentfirst, enumitem,fullpage,cleveref}
\usepackage[normalem]{ulem}

\theoremstyle{definition}
\newtheorem{sthm}{Theorem}[section]
\newtheorem*{nthm}{Theorem}
\newtheorem{slem}[sthm]{Lemma}
\newtheorem{scor}[sthm]{Corollary}

\newtheorem{srmk}[sthm]{Remark}

\newtheorem{sdef}[sthm]{Definition}
\newtheorem{question}{Question}

\newtheorem{theorem}{Theorem}

\newcommand{\mb}[1]{\mathbb{#1}}

\newcommand{\on}[1]{\operatorname{#1}}

\newcommand{\im}{\on{im}}

\renewcommand{\phi}{\varphi}

\newcommand{\ip}[1]{\left\langle #1 \right\rangle}

\newcommand{\Frac}{\on{Frac}}

\newcommand{\mbn}{\mb{N}}
\newcommand{\mbz}{\mb{Z}}
\newcommand{\depth}{\on{depth}}

\newcommand{\mfq}{\mathfrak{q}}
\newcommand{\mfm}{\mathfrak{m}}
\newcommand{\mfn}{\mathfrak{n}}

\newcommand{\hsl}{\operatorname{HSL}}
\newcommand{\fte}{\operatorname{Fte}}

\usepackage[symbol*,perpage]{footmisc}

\newcommand{\mcp}{\mathcal{P}}
\newcommand{\mca}{\mathcal{A}}

\newcommand{\mbm}{\mathbf{m}}
\newcommand{\mbe}{\mathbf{e}}
\newcommand{\mbf}{\mathbf{f}}
\newcommand{\mba}{\mathbf{a}}
\newcommand{\fbp}[1]{\left[ #1 \right]}
\begin{document}
\title{Homological properties of pinched Veronese rings}
\author[Kyle Maddox]{Kyle Maddox}
\address{Department of Mathematics, University of Kansas, 405 Snow Hall, 1460 Jayhawk Bvld, Lawrence, KS 66045, USA}
\email{maddox@ku.edu}
\author[Vaibhav Pandey]{Vaibhav Pandey}
\address{Department of Mathematics, University of Utah, 155 S 1400 E, Salt Lake City, UT 84112, USA}
\email{pandey@math.utah.edu}


\begin{abstract}
\noindent Pinched Veronese rings are formed by removing an algebra generator from a Veronese subring of a polynomial ring. We study the homological properties of such rings, including the Cohen-Macaulay, Gorenstein, and complete intersection properties. Greco and Martino classified Cohen-Macaulayness of pinched Veronese rings by the maximum entry of the exponent vector of the pinched monomial; we re-prove their results with semigroup methods and correct an omission of a small class of examples of Cohen-Macaulay pinched Veronese rings.  When the underlying field is of prime characteristic, we show that pinched Veronese rings exhibit a variety of F-singularities, including F-regular, F-injective, and F-nilpotent. We also compute upper bounds on the Frobenius test exponents of pinched Veronese rings, a computational invariant which controls the Frobenius closure of all parameter ideals simultaneously.
\end{abstract}

\maketitle

\section{Introduction}

\renewcommand{\thetheorem}{\Alph{theorem}}
A pinched Veronese ring is formed by removing any one of the algebra generators of a Veronese subring of a polynomial ring over a field. Let $k$ be a field and $\mbm=(m_1,\ldots,m_n)$ be an exponent vector in $\mbn^n$ with $m_1+\cdots +m_n=d$, we denote by $\mcp_{n,d,\mbm}$ the pinched Veronese ring in $n$ variables formed by removing the monomial generator $x_1^{m_1}\cdots x_n^{m_n}$ from the degree $d$ Veronese subring of $k[x_1,\ldots,x_n]$. For a vector $\mathbf{m}=(m_1,\ldots,m_n) \in \mbn^n$, we let $\max(\mathbf{m})= m_j$ if $m_j \ge m_i$ for all $1 \le i \le n$. 

These affine semigroup rings are rarely normal and provide examples of rings for which certain homological properties are difficult to prove. For instance, consider the pinched Veronese ring: 
\[R = k[x^3,y^3,z^3,x^2y,xy^2,x^2z,xz^2,y^2z,yz^2] \] formed by removing the generator $xyz$ of the third Veronese subring of $k[x,y,z]$, i.e. $R=\mcp_{3,3,(1,1,1)}$. It was not known whether $R$ was Koszul; the question was raised by Sturmfels in 1993 and it was settled in the affirmative in 2009 by Caviglia in \cite{GC}, who used the techniques of Gr\"obner bases and Koszul filtrations (See also \cite{Conca}).

The pinched Veronese ring $\mcp_{2,4,(2,2)}=k[x^4, x^3y, xy^3, y^4]$ is well known to not be Cohen-Macaulay (see \cite{Macaulay}). Both of the examples above raise two important questions which we attempt to address in this paper -- first, which pinched Veronese rings are Cohen-Macaulay, if any? Furthermore, can we understand the singularity types of these rings? 

In Section $2$ we discuss the normalizations of pinched Veronese rings and develop combinatorial tools which compare the affine semigroups defining pinched Veronese rings to those defining the corresponding Veronese rings. These combinatorial results serve as the technical heart of the paper and help to illuminate the homological properties of pinched Veronese rings in the later sections. Section $2$ also includes a brief discussion of the prime characteristic singularity types which arise for pinched Veronese rings. In particular, \Cref{lem: nilpotent cokernel F-rational} outlines when certain subrings of $F$-rational rings are $F$-nilpotent. 

In Section $3$, we use local cohomology and the combinatorial tools from Section $2$ to study the Cohen-Macaulay and Gorenstein properties of pinched Veronese rings. Indeed, we show that the above homological properties of these rings are sensitive to the algebra generator that is pinched out.

\begin{theorem}\label{thm A}
The pinched Veronese ring $\mathcal{P}_{n,d,\mathbf{m}}$ is Cohen-Macaulay if and only if one of the following three conditions hold.
\begin{itemize}
    \item $\max(\mathbf{m}) = d$.
    \item $n=2$ and $\max(\mathbf{m})= d-1$.
    \item $n=3$, $d=2$, and $\max(\mathbf{m}) = 1$.
\end{itemize}
Further, when $\max(\mathbf{m})= d-1$, $\mathcal{P}_{2,d,\mathbf{m}}$ is a Gorenstein ring with $a$-invariant zero, and when $\max(\mathbf{m}) = 1$, $\mathcal{P}_{3,2,\mathbf{m}}$ is a complete intersection ring.
\end{theorem}

This theorem re-proves the result \cite[Theorem~B]{GM} of Greco and Martino by using purely semigroup techniques and corrects an omission of a class of Cohen-Macaulay rings. In their original proof, Greco-Martino calculated the Betti numbers of pinched Veronese rings by means of the reduced homology of squarefree divisor complexes. 

The Cohen-Macaulay property of affine semigroup rings has been an area of active research since the 1970s. In \cite{Hoc}, Hochster proved the following famous result.
\begin{nthm}[Hochster, \cite{Hoc}]\label{Hochster}
If $Q$ is a normal semigroup, then the affine semigroup ring $k[Q]$ is Cohen-Macaulay.
\end{nthm}

In \cite{Trung-Hoa}, Trung and Hoa describe the Cohen-Macaulay property of an affine semigroup ring in terms of combinatorial and topological properties of the convex rational polyhedral cone spanned by the affine semigroup. 

In Section $4$, we show that when the characteristic of the underlying field is positive, pinched Veronese rings exhibit a variety of $F$-singularities. We show that a large class of pinched Veronese rings are $F$-nilpotent. Interestingly, a certain class of these rings are either $F$-nilpotent or $F$-injective depending on the characteristic of the field.

\begin{theorem}\label{thm B}
Let $k$ be a field of characteristic $p>0$. The $F$-singularity type of the pinched Veronese ring $\mcp_{n,d,\mathbf{m}}$ is as follows.
\begin{itemize}
    \item $\mcp_{n,d,\mathbf{m}}$ is $F$-regular for $\max(\mbm) = d$.
    \item When $d>2$, $\mcp_{n,d,\mathbf{m}}$ is $F$-nilpotent for $\max(\mbm) < d$. 
    \item $\mcp_{n,2,\mbm}$ is $F$-nilpotent if $p=2$ and $F$-injective if $p>2$. Further, $\mcp_{3,2,\mbm}$ is $F$-pure if $\max(\mbm)=1$ and $p>2$.
\end{itemize}
\end{theorem}

We also compute upper bounds on the Hartshorne-Speiser-Lyubeznik numbers and the Frobenius test exponents of pinched Veronese rings.

In Section $5$, we prove similar results for affine semigroup rings obtained by removing larger subsets of the Veronese generators. In particular, we consider removing any subset of algebra generators $x_1^{m_1}\cdots x_n^{m_n}$ with $\max(\mbm)<d-1$ from the Veronese subring and show that this situation is remarkably similar to the $\max(\mbm)<d-1$ case for single pinches. Finally, we conclude the paper with several open questions.

\section{Preliminaries}

Throughout, fix a field $k$ and let $n>1$ and $d>1$ be natural numbers. We let $S$ denote the standard graded polynomial ring $k[x_1,\ldots,x_n]$. For a vector $\mathbf{m}=(m_1,\ldots,m_n) \in \mbn^n$, we let $|\mathbf{m}|=\sum_{i=1}^n m_i$ and $\max(\mathbf{m})= m_j$ if $m_j \ge m_i$ for all $1 \le i \le n$. We say $\mathbf{m}$ is in \textbf{descending order} if $m_1 \ge m_2 \ge \ldots \ge m_n$. 

For a semigroup $A \subset \mbn^n$, we have the \textbf{affine semigroup ring} $k[A]=k[x_1^{a_1},\ldots,x_n^{a_n} \mid (a_1,\ldots,a_n)\in A]$. One well-understood example is given by the semigroup $\mca_{n,d}\subset \mbn^n$ generated by the set $\{ \mathbf{a} \in \mbn^n \mid |\mathbf{a}|=d\}$ and the corresponding affine semigroup ring $V_{n,d}=k[\mca_{n,d}]$, also known as the $d$th Veronese subring of the polynomial ring $S$. Since $V_{n,d}$ is a direct summand of $S$, many important ring theoretic properties descend from $k[x_1,\ldots,x_n]$ to $V_{n,d}$. These include the Cohen-Macaulay property, unique factorization property, and $F$-regularity in positive characteristic among others.

Now fix any $\mathbf{m} \in \mca_{n,d}$ with $|\mathbf{m}|=d$. Our primary object of interest throughout this paper is the semigroup $\mca_{n,d,\mathbf{m}}$ generated by $\{\mathbf{a} \in \mbn^n \mid |\mathbf{a}|=d \text{ and } \mathbf{a} \neq \mathbf{m}\}$ and the corresponding affine semigroup ring $\mcp_{n,d,\mathbf{m}} = k[\mca_{n,d,\mathbf{m}}]$, also known as the \textbf{pinched Veronese ring}.

\begin{slem}\label{lem:fraction field of pinched veroneses}
The field of fractions of $\mcp_{2,2,(1,1)}$ is $k(x^2, y^2)$. Otherwise, the field of fractions of $\mcp_{n,d,\mathbf{m}}$ is the same as that of $V_{n,d}$: \[ \text{Frac}(\mcp_{n,d,\mathbf{m}}) = \text{Frac}(V_{n,d}) =  k(x_2/x_n,x_3/x_n, \ldots, x_{n-1}/x_n, x_n^d). \]
\end{slem}

\begin{proof}
The first statement is clear since $\mcp_{2,2,(1,1)}$ is a regular ring. Now assume without loss of generality that $\mathbf{m}$ is arranged in descending order, and we observe that $x_i/x_1 \in \Frac(\mathcal{P}_{n,d,\mathbf{m}})$ for each $1 \le i \le n-1$: \[ \dfrac{x_i}{x_n} = \dfrac{x_ix_n^{d-1}}{x_n^d}\] where both the numerator and denominator are monomials in $\mathcal{P}_{n,d,\mathbf{m}}$. Since $\mathcal{P}_{n,d,\mathbf{m}}$ is an $n$-dimensional domain, and the claimed $k$-algebra generators of $\Frac(\mathcal{P}_{n,d,\mathbf{m}})$ are clearly algebraically independent, the result follows.
\end{proof}

\begin{srmk}
It follows immediately from the above lemma that the extension of fields $k \subset \text{Frac}(\mcp_{n,d,\mathbf{m}})$ is \emph{rational} for all choices of $n$, $d$, and $\mathbf{m}$.
\end{srmk}

\begin{sthm} \label{thm:normalization of PV rings}
The pinched Veronese ring $\mcp_{2,2,(1,1)}$ is normal. The normalizations of the other pinched Veronese rings $\mathcal{P}_{n,d,\mathbf{m}}$ are as follows.
\begin{align*}
\overline{\mathcal{P}_{n,d,\mathbf{m}}} = \left\{ \begin{array}{ll}
    V_{n,d} & \text{ if } \max(\mathbf{m}) < d \\
    \mathcal{P}_{n,d,\mathbf{m}} & \text{ if } \max(\mbm)=d \\
\end{array} \right.
\end{align*}
\end{sthm}

\begin{proof}
The first statement is clear since $\mcp_{2,2,(1,1)}$ is a regular ring. We may assume, without loss of generality, that $\mathbf{m}$ is in descending order. Denote $P = \mathcal{P}_{n,d,\mathbf{m}}$ and $V=V_{n,d}$. Begin with the case that $\max(\mbm)<d$.

Note that if $f=x_1^{m_1}\cdots x_n^{m_n} \in V$ then $f^d = (x_1^{d})^{m_1}\cdots (x_n^d)^{m_n} \in P$ (here we have used $\max(\mbm)<d$.) Thus $f$ satisfies the monic polynomial $T^d-f^d \in P[T]$. So, $P \subseteq V$ is an integral extension.

Now we show that $f \in \Frac(P)$. By \Cref{lem:fraction field of pinched veroneses}, we may write:
\[ f = \left( \dfrac{x_1}{x_n}\right)^{m_1} \left( \dfrac{x_2}{x_n}\right)^{m_2} \cdots \left( \dfrac{x_{n-1}}{x_n} \right)^{m_{n-1}}x_n^d \in \Frac(P).\] This shows that $V$ is contained in the normalization of $P$. Since $V$ is normal (as it is a direct summand of a polynomial ring), it follows that the normalization of $P$ is $V$.

To complete the proof, it suffices to show that $\mathcal{P}_{n,d,\mathbf{m}}$ is a normal semigroup ring for $\mathbf{m} = (d,0,\ldots, 0)$. Since the affine semigroup ring corresponding to a normal affine semigroup is normal, we show that the semigroup $\mca_{n,d,\mathbf{m}}$ is normal, i.e., we show that for a positive integer $m$,  if $m \mba  \in \mca_{n,d,\mathbf{m}}$ then $\mba \in \mca_{n,d,\mathbf{m}}$ by a contrapositive argument. So, assume that $\mathbf{a} \notin \mca_{n,d,\mathbf{m}}$. This happens precisely when $\mathbf{a} = (td-q, \mathbf{q})$ where $\mathbf{q}$ is a vector in $\mbn^{n-1}$ with $|\mathbf{q}| = q \leq t-1$. Then, $m\mba = (mtd - mq, m\mathbf{q})$. Notice that $m\mba \notin \mca_{n,d,\mathbf{m}}$ since: \[|m\mathbf{q}| = mq \leq mt-m \leq mt-1.\] This finishes the proof (cf. \cite[Theorem~3.1]{GM}). 
\end{proof}

\begin{srmk}
It follows immediately from Hochster's theorem and \Cref{thm:normalization of PV rings} that for $\max(\mathbf{m}) = d$, $\mathcal{P}_{n,d,\mathbf{m}}$ is a Cohen-Macaulay ring.
\end{srmk}

\subsection{An overview of prime characteristic singularities}

We are interested in exploring the singularities that pinched Veronese rings may have, and in the prime characteristic case we will see in Section 4 that nearly all pinched Veronese rings are \textit{$F$-nilpotent}. First, we need to establish some prime characteristic background. We will assume that all prime characteristic rings $R$ mentioned in this section are \textbf{$F$-finite}, that is, the Frobenius endomorphism $F:R\rightarrow R$ by $r\mapsto r^p$ is finite as a map of rings.

Fix $(R,\mfm)$ a local ring of prime characteristic $p>0$. The map $F:R\rightarrow R$ has been a central object of study in the singularity theory of prime characteristic rings and beyond since Ernst Kunz's famous theorem that $R$ is regular if and only if $F$ is a flat map (\cite[Theorem~2.1]{KunzRegular}). One of many consequences of flatness is that all ideals of a regular ring are Frobenius closed, as outlined below.

\begin{sdef} \label{def: fc/tc}
Let $R$ be a ring of prime characteristic $p>0$ and dimension $d$, and let $I\subset R$. The \textbf{$e$th Frobenius bracket power} $I^{\fbp{p^e}}$ of $I$ is the ideal generated by the set of elements $\{x^{p^e} \mid x \in I\}$. The \textbf{Frobenius closure} $I^F$ of $I$ is the ideal: \[ I^F=\left\lbrace x \in R \left| x^{p^e} \in I^{\fbp{p^e}}\right. \text{ for some } e \in \mbn \right\rbrace .\] The ideal $I$ is \textbf{Frobenius closed} if $I^F = I$.
\end{sdef}

Another important ideal closure operation in prime characteristic defined similarly is \textit{tight closure} $I^*$ of an ideal $I$, but we will not have need for it in this paper, and the definition is similar to that of the tight closure of the zero submodule given below. We refer the reader to \cite{HH} for the details of tight closure.

Two classical $F$-singularities are defined in terms of these ideal closure operations. 

\begin{sdef}
Let $R$ be a local ring of prime characteristic $p>0$. Then, if $R$ is a $\mbn$-graded local ring, $R$ is \textbf{$F$-regular} if all ideals of $R$ are tightly closed. If $R$ is reduced and excellent, $R$ is \textbf{$F$-pure} if all ideals of $R$ are Frobenius closed.
\end{sdef}

The notions of $F$-regularity and $F$-purity are defined outside of the cases given above, but many technicalities ensue. Since all of our prime characteristic rings will be standard graded $F$-finite domains, the definitions above suffice for us. For different notions of $F$-regularity and their equivalence in the graded case, we refer the reader to Lyubeznik-Smith \cite{LS}. For different notions of $F$-purity, we refer the reader to Hochster \cite{HochsterPurity}.

Other commonly studied $F$-singularity types are defined in terms of the \textit{Frobenius action} on local cohomology. In particular, the Frobenius map $F:R\rightarrow R$ induces a map $F:H^j_I(R)\rightarrow H^j_I(R)$ for any $j \in \mbn$, and $I\subset R$. It is important to note that this map is not $R$-linear, but $p$-linear, in that $F(x\xi) = x^p F(\xi)$ for any $x \in R$, $\xi \in H^j_I(R)$. We will focus on the case $I=\mfm$ when $(R,\mfm)$ is local.

\begin{sdef}
Let $(R,\mfm)$ be a local domain of prime characteristic $p>0$ and dimension $d$. The \textbf{tight closure of $0$ in $H^d_\mfm(R)$} is the $R$-submodule: \[
0^*_{H^d_\mfm(R)} = \{ \xi \in H^d_\mfm(R) \mid \text{there is a } 0 \neq c\in R \text{ with } cF^e(\xi)=0 \text{ for all } e \gg 0\}.\] A Cohen-Macaulay ring $R$ is \textbf{$F$-rational} if $0^*_{H^d_\mfm(R)}=0$.
\end{sdef}

Tight closure in this sense can be defined outside of the domain case but we will avoid the technicalities here. There are also other equivalent notions of $F$-rationality which make use of tight closure of parameter ideals. In particular, the two papers of Smith (\cite{Smith1994} and \cite{Smith2}) help outline the equivalence of these definitions.

\begin{srmk}\label{rmk: direct summand of polynomial ring is F-rational}
We note that $F$-rational rings are normal Cohen-Macaulay domains. Furthermore, if $R$ is a direct summand of a polynomial ring, then $R$ is $F$-regular, which implies it is $F$-rational. We refer the reader to Hochster and Huneke's seminal paper on tight closure \cite{HH} for an overview of these singularity types. 
\end{srmk}

Weaker than $F$-rationality is $F$-injectivity. In particular a local ring $(R,\mfm)$ of prime characteristic $p>0$ is \textbf{$F$-injective} if the Frobenius action $F:H^j_\mfm(R)\rightarrow H^j_\mfm(R)$ is injective for all $j$. Of recent interest have been singularity types opposite to $F$-injectivity. 

\begin{sdef}
Let $(R,\mfm)$ be a local ring of prime characteristic $p>0$ and dimension $d$. If for each $0 \le j < d$ we have that $F^e:H^j_\mfm(R)\rightarrow H^j_\mfm(R)$ is the zero map some $e \gg 0$, then $R$ is \textbf{weakly $F$-nilpotent}. If $R$ is weakly $F$-nilpotent and additionally $F^e(0^*_{H^d_\mfm(R)}) = 0$ for some $e \gg 0$, then $R$ is \textbf{$F$-nilpotent}.
\end{sdef}

\begin{srmk}
Notice that a local ring $(R,\mfm)$ is $F$-rational if and only if it is both $F$-nilpotent and $F$-injective. 
\end{srmk}

The definition of $F$-nilpotence was first given in Srinivas-Takagi \cite{ST} and an interesting ideal-theoretic characterization was given in Polstra-Quy \cite{PQ}, which we will explore in Section $4$.

\begin{srmk}
If $R$ is a graded ring, then $F$ is not a homogeneous map, but does multiply the degree of a homogeneous element by $p$. Consequently, if $\mfm$ is the homogeneous maximal ideal of $R$, the induced map $F:H^j_\mfm(R)\rightarrow H^j_\mfm(R)$ also multiplies the degree of a homogeneous class by $p$.
\end{srmk}

\begin{sdef}
Given a graded $R$-module $M$, a \textbf{graded Frobenius action} on $M$ is a $p$-linear map $f:M \rightarrow M$ so that $f([M]_n)\subset [M]_{np}$. We say $M$ is \textbf{nilpotent} if $f^e=0$ for some $e \in \mbn$. 
\end{sdef}

A graded Frobenius action $f:M\rightarrow M$ on an $R$-module $M$ induces another graded Frobenius action, which we also call $f$, on the graded local cohomology $H^j_I(M)$ for any homogeneous ideal $I\subset R$, and if $f:M\rightarrow M$ is nilpotent, so is $f:H^j_I(M) \rightarrow H^j_I(M)$ for all $j$. 

\begin{srmk}\label{rmk: homological algebra with graded frob actions}
Let $R$ be a graded ring of prime characteristic $p>0$ and suppose $A$, $B$, and $C$ are graded $R$-modules which fit into a commutative diagram as below: \begin{center}
    \begin{tikzcd}
        0\arrow{r} & A \arrow{d}{f_A} \arrow{r}{\alpha} & B \arrow{d}{f_B} \arrow{r}{\beta} & C \arrow{d}{f_C} \arrow{r} & 0 \\
        0\arrow{r} & A \arrow{r}{\alpha} & B \arrow{r}{\beta} & C \arrow{r} & 0
    \end{tikzcd}
\end{center} such that the rows are exact and the vertical maps are graded Frobenius actions. Then, for any homogeneous ideal $I\subset R$, we also get a commutative diagram of local cohomology modules: \begin{center}
    \begin{tikzcd}
        \cdots \arrow{r} & H^j_I(A) \arrow{d}{f_A} \arrow{r}& H^j_I(B) \arrow{d}{f_B} \arrow{r}{} & H^j_I(C) \arrow{d}{f_C} \arrow{r} & \cdots \\
        \cdots \arrow{r} & H^j_I(A) \arrow{r} & H^j_I(B) \arrow{r} & H^j_I(C) \arrow{r} & \cdots
    \end{tikzcd}
\end{center} where the rows are the induced long exact sequences of each row and the vertical maps are the induced Frobenius actions in each place. We refer the reader to \cite[Remark~2.2]{MM} for further details. 
\end{srmk}

We now conclude this subsection with a lemma of independent interest which we will apply to the normalization map $\mcp_{n,d,\mbm} \rightarrow V_{n,d}$. 

\begin{slem}\label{lem: nilpotent cokernel F-rational}
Let $(R,\mfm) \rightarrow (S,\mfn)$ be an inclusion of local rings (or graded rings with homogeneous maximal ideals $\mfm$ and $\mfn$ respectively) of prime characteristic $p>0$, and let $C$ be the cokernel of the map. Suppose that $S$ is $F$-rational and the induced Frobenius action $\overline{F}: C \rightarrow C$ is nilpotent. Then $R$ is $F$-nilpotent. 
\end{slem} 

\begin{proof}
Since $S$ is $F$-rational, $S$ is a domain, which implies $R$ is a domain. Further, notice that for some $e \in \mbn$ we have $\overline{F}^e:C \rightarrow C$ is the zero map implies for all $s \in S$, $s^{p^e} \in R$. We now also note that $\sqrt{\mfm S}=\mfn$, as if $s \in \mfn$, $s^{p^e} \in R$, and if $s^{p^e} \not \in \mfm \subset \mfm S$, then $s^{p^e}$ is a unit of $R$, and consequently a unit of $S$, contradicting that $s \in \mfn$. This also implies that $R$ and $S$ are both of the same dimension $n$.  

As in \Cref{rmk: homological algebra with graded frob actions}, associated to the short exact sequence: \begin{center}
\begin{tikzcd}
0 \arrow{r} & R \arrow{r} & S \arrow{r} & C \arrow{r} & 0
\end{tikzcd}\end{center} we obtain the long exact sequence: \begin{center}
\begin{tikzcd}
\cdots \arrow{r} & H^{j-1}_{\mfm}(C) \arrow{r} & H^j_{\mfm}(R) \arrow{r}& H^j_{\mfm}(S) \arrow{r} & H^j_\mfm(C) \arrow{r} & \cdots 
\end{tikzcd}
\end{center} in which we know $H^j_\mfm(S)=0$ for $j<n$ as $S$ is Cohen-Macaulay. Consequently, $H^{j-1}_\mfm(C) \simeq H^j_\mfm(R)$ for all $j<n$, and as $\overline{F}^e:H^{j-1}_\mfm(C)\rightarrow H^{j-1}_\mfm(C)$ is the zero map, we also have $F^e:H^j_\mfm(R)\rightarrow H^j_\mfm(R)$ is the zero map for $j<n$. Thus, $R$ is weakly $F$-nilpotent. 

To show that $R$ is $F$-nilpotent, let $\xi \in 0^*_{H^n_\mfm(R)}$ so that there is a nonzero $c \in R$ with $cF^e(\xi)=0$ for all $e \gg 0$. When $j=n$ in the long exact sequence above, we get an exact sequence: 
\begin{center}
\begin{tikzcd}
H^{n-1}_{\mfm}(C) \arrow{r}{\delta} & H^n_{\mfm}(R) \arrow{r}{\alpha}& H^n_{\mfm}(S)
\end{tikzcd}
\end{center} and for all $e \gg 0$ we have $\alpha(cF^e(\xi))=cF^e(\alpha(\xi))=0$ by \Cref{rmk: homological algebra with graded frob actions}. We also have $c$ is nonzero in $S$, so $\alpha(\xi)\in 0^*_{H^n_{\mfm}(S)}$, which implies $\alpha(\xi) = 0$ as $S$ is $F$-rational. 

Thus $\xi \in \im(\delta)$ by exactness, and we have $\xi=\delta(\xi')$ for some $\xi'\in H^{n-1}_\mfm(C)$. Since $C$ is nilpotent, we have $F^e(\xi)=F^e(\delta(\xi'))=\delta(\overline{F}^e(\xi'))=0$, and thus, $0^*_{H^n_{\mfm}(R)}$ is nilpotent. So, $R$ is $F$-nilpotent.
\end{proof} 

We note that the hypotheses of \Cref{lem: nilpotent cokernel F-rational} is satisfied when $R$ and $S$ are graded, and the cokernel $C$ has dimension $0$ as an $R$-module and is concentrated in positive degrees, since the induced map on $C$ multiplies the degree of a homogeneous element by $p$. 

\subsection{Key Combinatorial Results}

Our primary tool in the later sections is to control the cokernel $C$ of the natural inclusion of $\mcp_{n,d,\mbm}$ into its normalization $V_{n,d}$. Since this map corresponds to the natural ring homomorphism induced by the semigroup inclusion $\mca_{n,d,\mbm}\subset \mca_{n,d}$, we can understand the cokernel by computing the set difference $\mca_{n,d}\setminus \mca_{n,d,\mbm}$. 

Throughout this subsection, we let $T=T_{n,d}=\{\mbe \in \mbn^n \mid |\mbe|=d\}$ be the generating set for $\mca_{n,d}$ and for any fixed $\mbm \in T$, we let $T_{\mbm}=T\setminus \{\mbm\}$ be the generating set for $\mca_{n,d,\mbm}$. We will often re-write elements of $\mca_{n,d}$ in terms of elements of $\mca_{n,d,\mbm}$ and for brevity the following definition will be useful. 

\begin{sdef}
Let $\mba \in \mbz^n$. The \textbf{$(i,j)$-perturbation} of $\mba = (a_1,\ldots,a_i,\ldots,a_j,\ldots,a_n)$ is the vector: \[\mba'=(a_1,\ldots,a_i+1,\ldots,a_j-1,\ldots,a_n).\] We will refer to $\mathbf{b}$ as a perturbation of $\mba$ if $\mathbf{b}$ is the $(i,j)$-perturbation of $\mba$ for some $1\le i\le n$ and $1\le j\le n$.
\end{sdef}

\begin{srmk}
Notably, if $\mathbf{b}$ is a perturbation of $\mba$, $|\mathbf{b}|=|\mba|$, and in particular, $(i,j)$-perturbations of vectors in $\mca_{n,d}$ are still in $\mca_{n,d}$ as long as $j \neq 0$. Furthermore, if $\max(\mbm)<d$, then there is always a pair $(i,j)$ so that the $(i,j)$-perturbation and the $(j,i)$-perturbation of $\mbm$ is in $\mca_{n,d,\mbm}$. 
\end{srmk}

We also remark that when trying to re-write an element $\mbe$ of $\mca_{n,d}$ as a sum of vectors in $\mca_{n,d,\mbm}$, if $\mbe \neq \mbm + \mbf$ for any $\mbf \in \mca_{n,d}$ then $\mbe$ can be written purely as a sum of vectors in $T_{\mbm}$ and is thus in $\mca_{n,d,\mbm}$. We will implicitly use this fact in the proofs later in this subsection.

First, we consider the case $\max(\mbm)<d-1$ (for which we require $d>2$), where show that only point of $\mca_{n,d}$ missing from $\mca_{n,d,\mbm}$ is $\mbm$. 

\begin{slem}\label{lem: combinatorial cokernel max m < d-1 case}
Suppose $\mbm \in T_{n,d}$ has $\max(\mbm)<d-1$. Then $\mca_{n,d}\setminus \mca_{n,d,\mbm} = \{\mbm\}$.
\end{slem}

\begin{proof}
Without loss of generality, we assume $\mbm$ is in descending order. Let $\mbf \in \mca_{n,d}$ with $\mbf = \mbm + \mbe$ for some $\mbe \in \mca_{n,d}$. First, we show if $|\mbm+\mbe|=2d$, then $\mbm+\mbe \in A=\mca_{n,d,\mbm}$. We have that either $\max(\mbe)=d$ or $\mbe$ has at least two nonzero entries. 

In the former case, assume $e_1=d$. Then, letting $\mbm'$ be the $(1,2)$-perturbation of $\mbm$ and $\mbe'$ the $(2,1)$-perturbation of $\mbe$, we get $\mbm+\mbe = \mbm'+\mbe'$. If $e_i=d$ for some $i>1$, we can re-write $\mbm+\mbe=\mbm'+\mbe'$ similarly using the $(i,1)$-perturbation $\mbm'$ of $\mbm$ and the $(1,i)$-perturbation $\mbe'$ of $\mbe$. Since $\max(\mbm)<d-1$ and $\mbm$ is in descending order, all of the perturbations above are in $A$, we have shown $\mbm+\mbe\in A$ in any of the cases.

If $\mbe$ has at least two nonzero entries, say $e_i\neq 0$ and $e_j\neq 0$ with $i<j$, then we can re-write similarly. First, let $\mbm'$ be the $(i,1)$-perturbation of $\mbm$ and $\mbe'$ be the $(1,i)$-perturbation of $\mbe$, and then $\mbm+\mbe = \mbm'+\mbe'$ and $\mbe'$ is in $A$ unless $\mbe'=\mbm$. In this case, $e_i=m_i+1$ and $e_j=m_j$, and then instead of using $\mbe'$, we can re-write $\mbm+\mbe=\mbm'+\mbe''$ where $\mbm''$ is the $(j,1)$-perturbation of $\mbm$ and $\mbe''$ is the $(1,j)$-perturbation of $\mbe$. This shows that $A$ has every vector in the hyperplane $x_1+x_2+\cdots+x_n=2d$ of $\mbn^n$.

Finally, if $\mbe \in \mca_{n,d}$ with $|\mbe|=td$ for some $t>2$ and $t$ is even, we can write $\mbe$ as a sum of vectors in the $\sum x_i=2d$ hyperplane, which shows $\mbe \in A$. If $t$ is odd, $\mbe = \mbf+\mbe'$, where $|\mbe'|=(t-1)d$ and $|\mbf|=d$. Since $t-1$ is even, we can write $\mbe' = \mathbf{a}+\mathbf{b}$, where $\mba,\mathbf{b}\in A$ and $|\mathbf{b}|=(t-2)d$. Then $\mbe = \mbf + \mba + \mathbf{b}$, and $\mbf + \mba \in A$ since $|\mbf + \mba| = 2d$, thus $\mbe \in A$ also.
\end{proof}

Now we consider the $\max(\mbm)=d-1$ case, in which we show that the behavior depends on whether $d=2$ or $d>2$.

\begin{slem}\label{lem: combinatorial cokernel max m = d-1 case}
Suppose $\mbm \in T_{n,d}$ has $\max(\mbm)=d-1$, and without loss of generality we may assume $\mbm=(d-1,1,0,\ldots,0)$. Then, $\mca_{n,d}\setminus \mca_{n,d,\mbm}$ depends on $d$. \[\mca_{n,d}\setminus \mca_{n,d,\mbm} = \left\lbrace  \begin{array}{ll}
\{(2s+1,2t+1,0,\ldots,0) \mid s,t\ge 0\} & \text{ if } d=2 \\
\{(ds-1,1,0,\ldots,0) \mid s \ge 1\} & \text{ if } d>2 
\end{array}
\right.
\] 
\end{slem}

\begin{proof}
First, suppose $n=2$. If $d=2$, then $T_\mbm=\{(2,0),(0,2)\}$ thus $\mca_{2,2,(1,1)}$ consists of exactly the ordered pairs $(i,j)$ so that $i$ and $j$ are even, showing the result in this case. 

If $d>2$, then no vector $(d-j,j)$ of $T_{\mbm}$ has $j=1$. Consequently, all the vectors $(i,j)$ in $\mca_{2,d,\mbm}$ have $j\neq 1$, and so $\{(ds-1,1) \mid s \ge 1\} \subset \mca_{2,d}\setminus \mca_{2,d,\mbm}$. We now wish to show that if $j \neq 1$, then $\mbe = (i,j)$ is in $\mca_{2,d,\mbm}$. To see this, let $|\mbe|=td$ with $t>1$, and reduce $(i,j)$ mod $d$; if $i = ad+u$ and $j=bd+v$ with $0\le u,v\le d$ and $v \neq 1$, we can re-write $\mbe = a(d,0)+b(0,d)+(u,v)$, so $\mbe \in \mca_{2,d,\mbm}$. If $(u,v)=(d-1,1)$, notice we must have $b>0$. Then, write $\mbe = (ad+1,bd-1)+(d-2,2)$, and the previous case applies to $(ad+1,bd-1)$ since $bd-1 \neq 1$ mod $d$ when $d>2$.

This demonstrates all possible behavior in the plane. If $n>2$, the same arguments above work in the $x_1x_2$-plane to show that $\mca_{n,d}\setminus \mca_{n,d,\mbm}$ contains the given sets, and we will show that no other vectors of $\mca_{n,d}$ are missing from $\mca_{n,d,\mbm}$. 

If $d=2$, let $\mbe =(e_1,\ldots,e_n)\in \mca_{n,2}$ with $\mbe \neq (2s+1,2t+1,0,\ldots,0)$ for any $s,t \in \mbn$. We may assume $\mbe$ has $e_1$, $e_2$, and $e_i$ nonzero for some $i>2$. If both of $e_1$ and $e_2$ are even, the result is trivial. If $e_1$ is odd but $e_2$ is even, then we can re-write $\mbe$ using the $(1,i)$-perturbation to apply the previous case, and similarly if $e_2$ is odd but $e_1$ is even. So now we suppose $e_1$ and $e_2$ are odd, and if $e_i >1$ then we can use both the $(1,i)$- and $(2,i)$-perturbations to re-write $\mbe$ as a sum of vectors in $\mca_{n,2,\mbm}$. If $e_i=1$, then since $e_1,e_2$, and $e_i$ are all odd we must have a fourth nonzero entry $e_j$, and we can use both $(1,i)$- and $(2,j)$-perturbations to end up in previous cases, completing the proof when $d=2$.

If $d>2$, let $\mbe=(e_1,\ldots,e_n)\in \mca_{n,d}$ with $\mbe \neq (ds-1,1,0,\ldots,0)$ for any $s \in \mbn$. Suppose $|\mbe|=td$, and we will induce on $t$. When $t=1$, $\mbe \in T_\mbm$ and there is nothing to show. If $t=2$ and $e_i=0$ for $i>2$, $\mbe$ lies on the $x_1x_2$-plane of $\mbn^n$ and we can rely on the $n=2$ case to show $\mbe \in \mca_{n,d,\mbm}$. If $e_i \neq 0$ for some $i>2$, consider the possible values of $e_1$. If $e_1=d-1$, then we can re-write $\mbe$ as:\[
(d-1,e_2,\ldots,e_i,\ldots,e_n) = (d-1,0,\ldots,1,\ldots,0) + (0,e_2,\ldots,e_i-1,\ldots,e_n) \] and if $e_1\ge d$ then instead we can use: \[
(e_1,e_2,\ldots,e_i,\ldots,e_n)=(d,0,\ldots,0)+(e_1-d,e_2,\ldots,e_i,\ldots,e_n),
\] and in either case all of the vectors on the right hand side of the equations above are in $T_\mbm$. 

If $t>2$, then write $\mbe = \mbm + \mbf$, where $\mbf = (f_1,\ldots,f_n)$ must have $f_i>0$ for some $i>1$. If $i=2$, then we can re-write $\mbe=\mbm'+\mbf'$, where $\mbm'$ is the $(2,1)$-perturbation of $\mbm$ and $\mbf'$ is the $(1,2)$-perturbation of $\mbf$, unless $\mbf'=((t-1)d-1,1,0,\ldots,0)$. In this case, we have $f_1=(t-1)d-2$ and $f_2=2$, so instead we re-write $\mbe=\mbm''+\mbf''$ where $\mbm''$ is the $(1,2)$-perturbation of $\mbm$ and $\mbf''$ is the $(2,1)$-perturbation of $\mbf''$. In either case, by induction  $\mbf'$ (or $\mbf''$, as necessary) is in $\mca_{n,d,\mbm}$, and $\mbm'$ (or $\mbm''$) is in $T_\mbm$. If $i>2$, then we can similarly re-write using the $(i,2)$-perturbation of $\mbm$ and the $(2,i)$-perturbation $\mbf'$ of $\mbf$, unless $\mbf' = ((t-1)d-1,1,\ldots,0)$ and instead we can use the $(1,2)$-perturbation of $\mbm$ and the $(2,1)$-perturbation of $\mbf$. In either of these cases the result lives in $\mca_{n,d,\mbm}$. By induction, the proof is complete.
\end{proof}

These combinatorial results allow us to explicitly describe the cokernel of $\mcp_{n,d,\mbm}\rightarrow V_{n,d}$ whenever $\max(\mbm)<d$.

\begin{scor}\label{cor: cokernel is principal}
If $\max{\mbm}<d$, the cokernel $C$ of the inclusion of the pinched Veronese ring $P=\mcp_{n,d,\mbm}$ in its normalization $V=V_{n,d}$ is principally generated as a $P$-module by $x_1^{m_1}\cdots x_n^{m_n} + P$.
\end{scor}

\begin{proof}
If $\max(\mbm)<d-1$, certainly $y=x_1^{m_1}\cdots x_n^{m_n}$ is in $V\setminus P$ so forms a nonzero element in $C$. However, by \Cref{lem: combinatorial cokernel max m < d-1 case} we evidently have any nonzero element of $C$ is simply $\lambda y + P$ for any $\lambda \in k$, which shows $C$ is the $P$-module generated by $y+P$.

If $\max(\mbm)=d-1$, assume without loss of generality that $\mbm=(d-1,1,0,\ldots,0)$. If $d>2$, by \Cref{lem: combinatorial cokernel max m = d-1 case} any nonzero element of $C$ is a $k$-linear sum of monomials of the form $x_1^{ds-1}x_2 + P$ for some $s \ge 1$. However, $x_1^{ds-1}x_2 + P = (x_1^d)^{s-1}(x_1^{d-1}x_2+P)$, and so any element of $C$ can be written as $r(x_1^{d-1}x_2+P)$ for some $r \in P$. 

Similarly, if $\max(\mbm)=1$ and $d=2$, then any element of $C$ is a $k$-linear sum of monomials of the form $x_1^{2s+1}x_2^{2t+1}+P$ for some $s\in \mbn$ and $t \in \mbn$. However, $x_1^{2s+1}x_2^{2t+1}+P = (x_1^{2s}x_2^{2t})(x_1x_2+P)$, and so any element of $C$ can be written as $r(x_1x_2+P)$ for some $r \in P$.
\end{proof}

\section{The Cohen-Macaulay and Gorenstein properties}

Our aim in this section is to prove \Cref{thm A}.

\addtocounter{theorem}{-2}

\begin{theorem}
The pinched Veronese ring $\mathcal{P}_{n,d,\mathbf{m}}$ is Cohen-Macaulay if and only if one of the following three conditions hold.
\begin{itemize}
    \item $\max(\mathbf{m}) = d$.
    \item $n=2$ and $\max(\mathbf{m})= d-1$.
    \item $n=3$, $d=2$, and $\max(\mathbf{m}) = 1$.
\end{itemize}
Further, when $\max(\mathbf{m})= d-1$, $\mathcal{P}_{2,d,\mathbf{m}}$ is a Gorenstein ring with $a$-invariant zero, and when $\max(\mathbf{m}) = 1$, $\mathcal{P}_{3,2,\mathbf{m}}$ is a complete intersection ring.
\end{theorem}

We recover the result of Greco and Martino in \cite[Theorem~B]{GM} by using purely semigroup techniques and correct an omission of a class of Cohen-Macaulay rings. We prove \Cref{thm A} by studying the possible depths of the pinched Veronese rings using the combinatorial results from Section $2$.

\begin{slem} \label{lem: depth of pinched veroneses}
The depth of the pinched Veronese ring $\mathcal{P}_{n,d,\mathbf{m}}$ is as follows.
\begin{enumerate}
    \item If $\max(\mathbf{m}) = d$ then $\depth(\mathcal{P}_{n,d,\mathbf{m}}) = n$.
    \item If $n>2$, $d=2$, and $\max(\mathbf{m}) = 1$, then $\depth(\mathcal{P}_{n,d,\mathbf{m}}) = 3$.
    \item If $\max(\mathbf{m}) = d-1$, then $\depth(\mathcal{P}_{n,d,\mathbf{m}}) = 2$. Also, $\depth(\mathcal{P}_{2,2,\mathbf{m}}) = 2$.
    \item If $\max(\mathbf{m}) < d-1$ then $\depth(\mathcal{P}_{n,d,\mathbf{m}}) = 1$.
\end{enumerate}
\end{slem}
\begin{proof}
By \Cref{thm:normalization of PV rings}, when $\max(\mathbf{m}) = d$, $P = \mathcal{P}_{n,d,\mathbf{m}}$ is normal. By Hochster's theorem, normal affine semigroups are Cohen-Macaulay. Since $P$ is $n$ dimensional, it follows that the depth of $P$ is also $n$. This proves case $1$. 
    
Let $V= V_{n,d}$ denote the Veronese ring, which is the normalization of $P = \mcp_{n,d,\mbm}$ by \Cref{thm:normalization of PV rings}, and let $\mfm$ denote the homogeneous maximal ideal of $P$. We observe that the inclusion of affine semigroups $\mca_{n,d,\mbm}\subset \mca_{n,d}$ induces the short exact sequence of $P$-modules below.
   \begin{center}
\begin{tikzcd}
0 \arrow{r} & P \arrow{r} & V \arrow{r} & C \arrow{r} & 0
\end{tikzcd}\end{center} 
Let $d$ denote the depth of $P$. Since $V$ is a Cohen-Macaulay ring, applying the local cohomology functor supported at $\mfm$ to the above short exact sequence we get the inclusion $H^{d-1}_{\mfm} (C) \subset H^d_{\mfm} (P)$. 

In case $2$, the cokernel $C$ has depth $2$ since $x_1^2$ and $x_2^2$ form a maximal regular sequence on $C$ by \Cref{lem: combinatorial cokernel max m = d-1 case}. Since the first nonvanishing of the local cohomology module $H^i_{\mfm} (C)$ occurs at the depth of $C$ as a $P$-module, we get $H^2_{\mfm}(C)$ is nonzero. Consequently, $H^3_{\mfm}(P)$ is nonzero and the depth of $P$ is $3$.

In case $3$, for $n=2$, both $x_1^2$ and $x_2^2$ are regular on $C$ by the first part of \Cref{lem: combinatorial cokernel max m = d-1 case} so that $P$ has depth $2$. For $n > 2$, since $x_1^d$ is a maximal regular sequence on $C$ by the second part of \Cref{lem: combinatorial cokernel max m = d-1 case}, so the depth of $P$ is again $2$. Further, $\mcp_{2,2,(1,1)}\simeq k[x^2,y^2]$ is regular and hence depth $2$ as well.

In case $4$, the cokernel $C$ is an Artinian $P$-module by \Cref{lem: combinatorial cokernel max m < d-1 case} so $P$ has depth $1$.
\end{proof}

The lemma above, together with the fact that $\dim(\mcp_{n,d,\mbm})=n$, settles the first part of \Cref{thm A}. 

\begin{srmk}
 The above argument shows that when $\max(\mbm)<d-1$, all the lower local cohomology modules of $\mcp_{n,d,\mbm}$ are finitely generated, and thus of finite length. Therefore $\mcp_{n,d,\mbm}$ is a \textbf{generalized Cohen-Macaulay} ring when $\max(\mbm)<d-1$.
\end{srmk}

When $\max(\mbm) =d$, $\mcp_{n,d,\mathbf{m}}$ is Cohen-Macaulay. In dimension $2$ we can be more explicit.

\begin{srmk}
If $\max(\mbm) =d$, then $\mcp_{2,d,\mathbf{m}}$ is isomorphic to the Veronese ring $V_{2,d-1}$. We assume $\mathbf{m}=(d,0)$ and let $P$ be the sub-semigroup of $\mbz^2$ generated by $\{(d-1,1),(d-2,2),\ldots,(1,d-1),(0,d)\}$. If $\ip{P}$ is the subgroup of generated by $P$ in $\mbz^2$, we see $\ip{P}=\ip{(0,d),(1,-1)}$, and these generators are $\mbz$-linearly independent. Consequently $\ip{P}$ is a free abelian group of rank 2.

 We define a lattice isomorphism $\ip{P}\simeq \mbz^2$ given by $(0,d)\mapsto (1,0)$ and $(1,-1) \mapsto (0,1)$. Under this transformation, the points $ (d-1,1),(d-2,2),\ldots,(0,d)$ map to $(1,d-1),(1,d-2),\ldots,(1,0)$ respectively. We let $Q$ be the sub-semigroup of $\mbz^2$ generated by these points. 

We see that $Q$ is isomorphic as a semigroup to the normal semigroup $\mathcal{A}_{2,d-1}$ and the semigroup isomorphism $\mca_{2,d,\mathbf{m}} \simeq \mathcal{A}_{2,d-1}$ induces a ring isomorphism $P \simeq k[\mathcal{A}_{2,d-1}] = V_{2,d-1}$.
\end{srmk}

We now study the Gorenstein property of pinched Veronese rings. Recall for a graded ring $R$ of dimension $d$ with homogeneous maximal ideal $\mathfrak{m}$, Goto and Watanabe in \cite{GW} define the $a$-invariant of $R$ to be the highest integer $a(R)=a$ such that the grade piece $[H^d_{\mathfrak{m}}(R)]_a$ is nonzero. 

\begin{srmk}
If $\max(\mbm)=d$, then $P = \mcp_{2,d,\mathbf{m}}$ is Gorenstein if and only if $d=2,3$. This is because when $\max(\mbm)=d$, then $P \simeq V_{2,d-1}$ which is a Veronese subring of $S=k[x,y]$. $S$ is a regular (hence Gorenstein) ring with $a(S)=-2$, so by \cite{GW}, Corollary 3.1.5 and Theorem 3.2.1, $V_{2,d-1}=S^{(d-1)}$ is Gorenstein if and only if $d-1\mid 2$. Thus $d=2,3$. 
\end{srmk}

\begin{sthm}\label{thm:max = d-1 is gorenstein}
Let $\max(\mbm)=d-1$. Then $\mcp_{2,d,\mathbf{m}}$ is a Gorenstein ring of $a$-invariant zero. When $\max(\mathbf{m}) = 1$, $\mathcal{P}_{3,2,\mathbf{m}}$ is a complete intersection ring.
\end{sthm}

\begin{proof}
Without loss of generality, we may assume that $\mathbf{m} = (d-1,1)$. By \Cref{lem: depth of pinched veroneses}, $P = \mcp_{2,d,\mathbf{m}}$ is Cohen-Macaulay. Thus, the system of parameters $x^d, y^d$ is a regular sequence in $P$. Consequently, it is enough to show that the zero-dimensional ring $R = P/(x^d, y^d)$ is Gorenstein. We claim that:
\[R = k \oplus kx^{d-2}y^2 \oplus \cdots \oplus kx^2y^{d-2} \oplus kxy^{d-1} \oplus kx^{d-1}y^{d+1}.\]
This is because for $i+j \neq d-1$, $(x^iy^{d-i})(x^jy^{d-j})$ lies in the ideal $(x^d,y^d)$, as shown by the computation below.

\begin{align*}
(x^iy^{d-i})(x^jy^{d-j}) = \left\{ \begin{array}{cc} 
                x^d(x^{i+j-d}y^{2d-(i+j)}) & \hspace{5mm} \text{if } i+j \geq d \\
                y^d(x^{i+j}y^{d-(i+j)}) & \hspace{5mm} i+j<d \\
                \end{array} \right.
\end{align*}

Thus the socle of $R$ is singly generated by $x^{d-1}y^{d+1}$. Therefore $P$ is Gorenstein.

Further, notice that the (nonzero) element of maximal degree in the top local cohomology module of $P$ supported at its homogeneous maximal ideal is:
\[
\eta =\left[ \frac{x^{d-1}y^{d+1}}{x^dy^d}\right].
\]
Under the standard grading, the degree of $\eta$ is $(d-1)+(d+1) -2d = 0$. So, the $a$-invariant of $P$ is zero.

Next, it suffices to observe that $P = \mathcal{P}_{3,2,(1,1,0)}$ is a complete intersection ring. Consider the surjection:
\[\phi :R = k [a,b,c,d,e] \rightarrow  P = k[x^2,xz,y^2,yz,z^2]\] given by $a \mapsto x^2$, $\ldots$, $e \mapsto z^2$. 

The ideal $I =(ae-b^2, ce-d^2)$ is clearly contained in the kernel of $\phi$ and defines a prime ideal generated by a regular sequence of length two in the ring $R$. Since $R/I$ is a three dimensional domain, it follows that $\phi$ is an isomorphism of rings so that $P$ is a complete intersection.
\end{proof}

\section{Pinched Veronese Rings in Prime Characteristic}

In this section, we will explore the $F$-singularities of pinched Veronese rings which are of prime characteristic $p>0$. As an application, we will show that all pinched Veronese rings have finite \textit{Frobenius test exponents}, which is an invariant that controls the Frobenius closure of all parameter ideals simultaneously. We again assume that the underlying field $k$ is $F$-finite, which by a theorem of Kunz \cite[Theorem~2.5]{KunzExcellent}, implies any finite type $k$-algebra is excellent.

\subsection{F-nilpotence of Pinched Veroneses}

Our aim in this subsection is to prove \Cref{thm B} by cases on $\max(\mbm)$. 

\begin{theorem}
Let $k$ be a field of characteristic $p>0$. The $F$-singularity type of the pinched Veronese ring $\mcp_{n,d,\mathbf{m}}$ is as follows.
\begin{itemize}
    \item $\mcp_{n,d,\mathbf{m}}$ is $F$-regular for $\max(\mbm) = d$.
    \item $\mcp_{n,d,\mathbf{m}}$ is $F$-nilpotent for $d>2$ and $\max(\mbm) < d$. 
    \item $\mcp_{n,2,\mbm}$ is $F$-nilpotent if $p=2$ and $F$-injective if $p>2$. Further, $\mcp_{3,2,\mbm}$ is $F$-pure if $\max(\mbm)=1$ and $p>2$.
\end{itemize}
\end{theorem}

\begin{srmk}
If $\max(\mbm)=d$, then $\mcp_{n,d,\mbm}$ is a direct summand of a polynomial ring, and is hence $F$-regular, as noted before.
\end{srmk}

When $\max(\mbm)<d$, we will use \Cref{lem: nilpotent cokernel F-rational} to conclude that $P_{n,d,\mbm}$ is $F$-nilpotent. 

\begin{sthm}
If $d>2$ and $\max(\mbm)<d$, then the pinched Veronese ring $\mcp_{n,d,\mbm}$ is $F$-nilpotent. 
\end{sthm}

\begin{proof}
Write $P=\mcp_{n,d,\mbm}$. First, suppose $\max(\mbm)=d-1$, and assume without loss of generality that $\mbm=(d-1,1,0,\ldots,0)$. If $d=2$, then $P \simeq k[x^2,y^2]$ is regular and is thus $F$-nilpotent. If $d>2$, By \Cref{lem: combinatorial cokernel max m = d-1 case}, the cokernel $C$ of $P\rightarrow V_{n,d}$ is principally generated as a $P$-module by $x_1^{d-1}x_2+P$. Then, $\overline{F}(x_1^{d-1}x_2 + P) = x_1^{dp-p}x_2^p + P = P$ since the exponent on $x_2$ is not 1. Consequently, $\overline{F}:C \rightarrow C$ is the zero map.

Now assume $\max(\mbm)<d-1$. Then, by \Cref{lem: combinatorial cokernel max m < d-1 case}, we know the cokernel $C$ of $P\rightarrow V_{n,d}$ is concentrated in a single positive degree, and so $\overline{F}(C)=0$ as well, since $\overline{F}$ multiplies degrees by $p$.

In either case, we can apply \Cref{lem: nilpotent cokernel F-rational} to conclude the proof. 
\end{proof}

To handle the $d=2$ case, we will need to utilize the Nagel-Schenzel isomorphism given in \cite{NS} for local cohomology modules. We re-state the isomorphism below in the context that we need -- the original statement utilizes \textit{filter-regular sequences}, which are a generalization of regular sequences.

\begin{nthm}[Nagel-Schenzel]
Let $(R,\mfm)$ be a local ring and let $x_1,\ldots,x_t\in \mfm$ be a regular sequence on a finitely-generated $R$-module $M$, with $I=(x_1,\ldots,x_t)$. Then, $H^t_\mfm(M)\simeq H^0_\mfm(H^t_I(M))$. 
\end{nthm}

We now show the $F$-singularity type in the $d=2$ case depends on the parity of $p$.

\begin{sthm}
Let $k$ be a field of characteristic $p>0$. The pinched Veronese ring $\mcp_{n,2,\mbm}$ with $\max(\mbm)=1$ is $F$-nilpotent for $p = 2$ and is $F$-injective for $p > 2$. 
\end{sthm}

\begin{proof} 
Assume without loss of generality that $\mbm=(1,1,0,\ldots,0)$. We saw in \Cref{lem: combinatorial cokernel max m = d-1 case} that $C$ is principally generated as a $P$-module by $x_1x_2+P$. If $p=2$, $\overline{F}(x_1x_2+P) = x_1^2x_2^2+P = P$ since $x_1^2$ and $x_2^2$ are in $P$. Hence, if $p=2$, $\overline{F}:C \rightarrow C$ is the zero map. We can now apply \Cref{lem: nilpotent cokernel F-rational} to see that $P$ is $F$-nilpotent.

If $p$ is odd, however, then for any nonzero  $r\cdot x_1x_2 +P\in C$, we have $\overline{F}(r\cdot x_1x_2+P) = r^p\cdot x_1^px_2^p+P$, and since $p$ is odd, $x_1^px_2^p \not \in P$. Furthermore, we can see that $x_1^2$ and $x_2^2$ are regular elements on $C$ but all other generators of the homogeneous maximal ideal $\mfm$ of $P$ annihilate $C$, so that if $r\cdot x_1x_2+P \neq P$ then $r^p\cdot x_1^px_2^p+P \neq P$ as well, as $r$ can only be of the form $\lambda x_1^{2i}x_2^{2j}$ for some $\lambda \in k$ and $i,j \in \mbn$. Consequently, $\overline{F}:C \rightarrow C$ is injective as well. 

Now we will now show that the induced Frobenius map $\overline{F}$ on $H^2_\mfm(C)$ is also injective. Letting $I=(x_1^2,x_2^2)$, by the Nagel-Schenzel isomorphism we have $H^2_\mfm(C) = H^0_\mfm(H^2_I(C))$. We consider $H^2_I(C)$ as the direct limit: \[\varinjlim_i \left( C/(x_1^{2i},x_2^{2i})C \xrightarrow{\cdot x_1^2x_2^2} C/(x_1^{2i+2},x_2^{2i+2})C \right).\]

If $[z + (x_1^{2i},x_2^{2i})C]$ is a class in the direct limit, then $\overline{F}([z + (x_1^{2i},x_2^{2i})C]) = [\overline{F}(z) +(x_1^{2ip},x_2^{2ip})C]$. As $x_1^2$ and $x_2^2$ form a regular sequence on $C$, the direct limit system is injective, and consequently $[\overline{F}(z) + (x_1^{2ip},x_2^{2ip})C]=0$ if and only if $\overline{F}(z) \in (x_1^{2ip},x_2^{2ip})C$. However, if $z \not \in (x_1^{2i},x_2^{2i})C$, then $z$ is a $k$-linear sum of monomials of the form $x_1^{2s+1}x_2^{2t+1}+P$ with $s<i$ and $t<i$. Then, $\overline{F}(z)$ is a $k$-linear sum of monomials of the form $x_1^{2sp+p}x_2^{2tp+p}+P$. We can see that this is not in $(x_1^{2ip},x_2^{2ip})C$ as $s<i$ implies $2s+1<2i$, so $2sp+p<2ip$, and similarly, $2tp+p<2ip$. Thus, $\overline{F}$ is injective on $\varinjlim_i C/(x_1^{2i},x_2^{2i})C = H^2_I(C)$, and so $\overline{F}$ is injective on the submodule $(0:_{H^2_I(C)} \mfm^{\infty}) \simeq H^2_\mfm(C)$. 

If $n=3$, to see that $P$ is $F$-injective we need only see that the Frobenius action on $H^3_\mfm(P)$ is injective as $P$ is Cohen-Macaulay. We have the short exact sequence: \begin{center}
    \begin{tikzcd}
    0 \arrow{r} & H^2_\mfm(C) \arrow{r} & H^3_\mfm(P) \arrow{r} & H^3_\mfm(V) \arrow{r}& 0
    \end{tikzcd}
\end{center} and the Frobenius actions on the outer two modules are injective, which implies the Frobenius action on $H^3_\mfm(P)$ is also. If $n>3$, we get that the local cohomology of $P$ is completely described by the isomorphisms $H^2_\mfm(C)\simeq H^3_\mfm(P)$ and $H^n_\mfm(P)\simeq H^n_\mfm(V)$, and since $\overline{F}$ on $H^2_\mfm(C)$ and $F$ on $H^n_\mfm(V)$ are injective, $P$ is $F$-injective.
\end{proof}

\begin{scor}\label{thm: omitted example is F-pure}
By \Cref{thm A}, we have $\mcp_{3,2,\mbm}$ with $\max(\mbm)=1$ is Gorenstein, and the previous theorem shows that it is also $F$-injective in odd characteristic. Hence, by a result of Fedder (\cite[Lemma~3.3]{Fedder}), $\mcp_{3,2,\mbm}$ with $\max(\mbm)=1$ is $F$-pure in odd characteristic.
\end{scor}

\begin{srmk}
As noted in Section 2, $F$-rational rings are both $F$-nilpotent and $F$-injective. Interestingly, the class of pinched Veronese rings in the above theorem are never $F$-rational (since they are not normal), but may have either $F$-nilpotent or $F$-injective singularities depending on the characteristic of the field. 
\end{srmk}

 In \cite[Example 2.5]{pandey}, the second author calculates the cohomological dimension of the ideal defining $\mathcal{P}_{(2,4,(2,2)}$ over the integers by using the normalization map. We have now finished the proof of \Cref{thm B}. 

\subsection{Frobenius Test Exponents for Pinched Veroneses}

In this subsection, we explore the Frobenius test exponent for parameter ideals of pinched Veronese rings. This numerical invariant measures how far any parameter ideal is from being Frobenius closed in the following sense. 

For any ideal $I$, $I^F$ is finitely generated, so there is an $e \in \mbn$ so that $(I^F)^{\fbp{p^e}} = I^{\fbp{p^e}}$. If we restrict to the class of parameter ideals, there may possibly be a uniform exponent which trivializes the Frobenius closure of all parameter ideals simultaneously, which motivates the following definition.

\begin{sdef}
Let $(R,\mfm)$ be a (graded) local ring of prime characteristic $p>0$ and let $\mfq$ be an ideal generated by a full (homogeneous) system of parameters of $R$. Then, the \textbf{Frobenius test exponent of $\mfq$} is the smallest $e \in \mbn$ so that $(\mfq^F)^{\fbp{p^e}}=\mfq^{\fbp{p^e}}$. The \textbf{Frobenius test exponent for $R$} is: \[\fte(R)=\sup\{\fte(\mfq) \mid \mfq \text{ is a (homogeneous) parameter ideal of } R\}.\]
\end{sdef}

In this sense, $\fte(R)$ uniformly annihilates Frobenius closure relations for all parameter ideals. Frobenius test exponents have been shown to be finite in several important cases, typically related to nilpotence properties. See the introduction of \cite{Mad19} for a historical survey of finiteness for Frobenius test exponents. In particular, Katzman-Sharp showed in \cite[Theorem~2.4]{KS} that $\fte(R)=\hsl(H^{\dim(R)}_\mfm(R))$ (defined below) when $R$ is Cohen-Macaulay. Recently, Quy showed in \cite{Q} that weakly $F$-nilpotent rings have finite Frobenius test exponent as well. We will utilize Quy's upper bound to bound Frobenius test exponents for pinched Veronese rings.

When $R$ is $F$-nilpotent, Polstra-Quy show in \cite[Theorem~5.11]{PQ} that $\mfq^*=\mfq^F$ for all parameter ideals $\mfq$ of $R$. Since nearly all pinched Veronese rings are $F$-nilpotent, we can also treat the Frobenius test exponent as a measure of how far parameter ideals in these rings are from being tightly closed, as letting $e = \fte(R)$, then $(\mfq^*)^{\fbp{p^e}} = \mfq^{\fbp{p^e}}$.

Upper bounds for Frobenius test exponents are typically given in terms of the Hartshorne-Speiser-Lyubeznik numbers of $R$.

\begin{sdef}
Let $(R,\mfm)$ be a graded local ring of prime characteristic $p>0$ and let $0 \le j \le d = \dim(R)$. Write $0^F_{H^j_\mfm(R)}$ for the set of elements of the graded local cohomology module $H^j_\mfm(R)$ which are in the kernel of $F^e$ for some $e$. Then the \textbf{Hartshorne-Speiser-Lyubeznik number of $H^j_\mfm(R)$} is defined as: \[ \hsl(H^j_\mfm(R)) = \inf\left\lbrace e \in \mbn \left| F^e\left(0^F_{H^j_\mfm(R)}\right)\right. = 0 \right\rbrace. \] Furthermore, the \textbf{Hartshorne-Speiser-Lyubeznik number of $R$} is defined as: \[ \hsl(R)=\sup \{\hsl(H^j_\mfm(R)) \mid 0 \le j \le d\}.\]
\end{sdef}

Amazingly, the Hartshorne-Speiser-Lyubeznik numbers of $R$ must be finite since $H^j_\mfm(R)$ is Artinian. Initial results about finiteness for these numbers is due to Hartshorne and Speiser, and hypotheses were removed by Lyubeznik and later Sharp. Their results are collected below.

\begin{nthm}[Hartshorne-Speiser \cite{HS}, Lyubzenik \cite{Lyubeznik1997}, Sharp \cite{Sharp07}]
Let $(R,\mfm)$ be a (graded) local ring of prime characteristic $p>0$. Then, $\hsl(R)<\infty$.
\end{nthm}

Notice that $R$ is $F$-injective if and only if $\hsl(R)=0$. Since $V=V_{n,d}$ is $F$-regular, it is in particular $F$-injective, and so $\hsl(V)=0$. We can then find upper bounds on the Hartshorne-Speiser-Lyubeznik numbers of $P=\mcp_{n,d,\mbm}$ using the cokernel of $P\rightarrow V$.

\begin{sthm}\label{thm: hsl numbers for pinched veronese in 2 vars}
The pinched Veronese ring $\mcp_{n,d,\mbm}$ has $\hsl(P)\le 1$.
\end{sthm}

\begin{proof}
Let $P=\mcp_{n,d,\mbm}$, $V = V_{n,d}$, $\mfm$ be the homogeneous maximal ideal of $P$, and let $C$ be the cokernel of $P\rightarrow V$. We now analyze the Hartshorne-Speiser-Lyubeznik numbers of $P$. 

Three simple cases are when $\max(\mbm)=d$, when $n=d=2$ and $\mbm=(1,1)$, and finally when $d=2$, $\max(\mbm)=1$ and $p>2$, since in these cases, $P$ is $F$-injective, so $\hsl(P)=0$.

In the remaining cases, we have seen that the Frobenius action $\overline{F}: C\rightarrow C$ is the zero map, which implies that the induced Frobenius action $\overline{F}:H^j_\mfm(C) \rightarrow H^j_\mfm(C)$ is the zero map as well. Notably, as observed in the proof of \Cref{lem: depth of pinched veroneses}, in all cases $C$ is a Cohen-Macaulay $P$-module.

In the cases where $P$ itself is not Cohen-Macaulay, write $\depth(P)=j$. We then get $H^{j-1}_\mfm(C)\simeq H^j_\mfm(P)$, and in particular, we also have  $F:H^j_\mfm(P)\rightarrow H^j_\mfm(P)$ is also the zero map. This implies $\hsl(H^j_\mfm(P))=1$. The only other local cohomology module for $P$ is $H^n_\mfm(P)\simeq H^n_\mfm(V)$, and since $F:H^n_\mfm(V) \rightarrow H^n_\mfm(V)$ is injective, we have $\hsl(H^n_\mfm(P))=0$. This shows $\hsl(P)=1$ when $P$ is not Cohen-Macaulay.

When $P$ is Cohen-Macaulay but not $F$-injective, then we get the short exact sequence: \begin{center} \begin{tikzcd}
0 \arrow{r} & H^{n-1}_\mfm(C) \arrow{r}{\delta} & H^n_\mfm(P) \arrow{r}{\alpha} & H^n_\mfm(V) \arrow{r} & 0
\end{tikzcd}\end{center} and for any $\xi \in H^n_\mfm(P)$ such that $F^e(\xi)=0$, we have $\alpha(F^e(\xi))=F^e(\alpha(\xi))=0$. Since $F^e$ is injective on $H^n_\mfm(V)$, we know $\alpha(\xi)=0$ and $\xi = \delta(\xi')$ for some $\xi' \in H^{n-1}_\mfm(C)$. Then, $F(\xi) = F(\delta(\xi')) = \delta(\overline{F}(\xi')) = 0$, which shows $\hsl(H^n_\mfm(P)) = 1$ and $\hsl(P)=1$.  
\end{proof}

Now we will utilize the primary result of \cite{Q}, where Quy shows that the Hartshorne-Speiser-Lyubeznik numbers bound the Frobenius test exponent.

\begin{nthm}[Quy]
Let $(R,\mfm)$ be a local ring of prime characteristic $p>0$ and dimension $n$. If $R$ is weakly $F$-nilpotent, then: \[\fte(R) \le \sum_{j=0}^n \binom{n}{j}\hsl(H^j_\mfm(R)).\]
\end{nthm}

Quy's result translates to the graded local setting with minimal adjustment, so we can apply it here to $\mcp_{n,d,\mbm}$.

\begin{scor}\label{cor: fte for pinched veroneses}
The pinched Veronese ring $P=\mcp_{n,d,\mathbf{m}}$ has an upper bound on its Frobenius test exponents given below. \begin{itemize}
    \item In the cases that $\max(\mbm)=d$; $n=d=2$ and $\max(\mbm)=1$; or $n=3$, $d=2$, $\max(\mbm)=1$ and $p> 2$, then $\fte(P)=0$, i.e., every parameter ideal of $P$ is Frobenius closed.
    \item If $d>2$ and $\max(\mbm)=d-1$, then $\fte(P)\le \binom{n}{2}$. In particular, when $n=2$, $\fte(P)=1$.
    \item If $d=2$, $n\ge 3$, $\max(\mbm)=1$, and $p=2$, then $\fte(P) \le \binom{n}{3}$. In particular, when $n=3$, $\fte(P)=1$.
    \item If $\max(\mbm)<d-1$, then $\fte(P)\le n$.
\end{itemize}
\end{scor}

\begin{proof}
We handle each case in turn. In the first case, no matter the conditions on $n,d,$ and $\mbm$, we have that $P$ is Cohen-Macaulay and $F$-injective, so by Quy-Shimomoto \cite[Corollary~3.9]{QS}, $\fte(P)=0$. 

Now we let $\mfm$ be the homogeneous maximal ideal of $P$. If $d>2$ and $\max(\mbm)=d-1$, we can apply Quy's upper bound to get: \[
\fte(P) \le \sum_{j=0}^n \binom{n}{j} \hsl(H^j_\mfm(P)) = \binom{n}{2} \hsl(H^2_\mfm(P)) = \binom{n}{2}
\] and if $n=2$, then $\fte(P)\le 1$. But in this case, $P$ is Cohen-Macaulay, and Katzman-Sharp show that $\fte(P)=\hsl(P)$ when $P$ is Cohen-Macaulay as mentioned above.

If $d=2$, $\max(\mbm)=1$, and $p=2$, the situation is very similar to the $\max(\mbm)=d-1$ case. When $n=3$, $P$ is Cohen-Macaulay and $\fte(P)=\hsl(P)=1$, and when $n>3$ then $\fte(P) \le \binom{n}{3} \hsl(H^3_\mfm(P)) + \binom{n}{n}\hsl(H^n_\mfm(P)) = \binom{n}{3}$. 

Finally, if $\max(\mbm)<d-1$, then Quy's upper bound gives: \[
\fte(P) \le \sum_{j=0}^n \binom{n}{j} \hsl(H^j_\mfm(P)) = \binom{n}{1} \hsl(H^1_\mfm(P)) = n,
\] as required.
\end{proof}

\begin{srmk}
Notably missing from the list above is the case that $n>3$, $d=2$, $\max(\mbm)=1$, and $p>2$. In this case, $P=\mcp_{n,2,\mbm}$ is $F$-injective but not $F$-nilpotent or Cohen-Macaulay, so none of the techniques used in \Cref{cor: fte for pinched veroneses} apply. The authors currently do not know if these examples have finite Frobenius test exponents. 
\end{srmk}

\section{Multi-pinched Veronese rings}

Throughout this section, let $T=T_{n,d}$ be as before with $d>2$ and $n \ge 2$. Most of the results in this paper so far have focused semigroups generated by removing only a single vector of $T$. However, we show that even if we remove a larger subset of $T$, we can still control the set difference between $\mca_{n,d}$ and the corresponding semigroup as long as we do not remove any vector $\mbe$ of $T$ with $\max(\mbe)\ge d-1$. The authors are grateful to Mark Denker for his combinatorial insight and inspiration for the lemma and proof below.

\begin{slem}\label{lem: combinatorial cokernel multipinch case}
We let $S\subset T$ be the set of vectors in $\mbn^n$ below. \[ S = \{ \mbm \in T \mid \max(\mbm)\ge d-1\}\] Let $A$ be the semigroup generated by $S$. Then, $\mca_{n,d}\setminus A$ is a finite set. In particular, if $\mathbf{e} \in \mca_{n,d}$ and $\max(\mathbf{e}) \ge (n-1)(d^2-d)$, then $\mathbf{e} \in A$. 
\end{slem}

\begin{proof}
We will show that if $t\ge (n-1)(d^2-d)$, then $\mbe = (t,e_2,\ldots,e_n)$ is in $A$, which will imply the result since the argument is symmetric in each coordinate. Using the division algorithm, write $e_i = dp_i+q_i$. We can re-write $\mbe = d(0,p_2,\ldots,p_n)+(t,q_2,\ldots,q_n)$ with $d(0,p_2,\ldots,p_n) \in A$. Then, for $2 \le i \le n$, we let $\mathbf{a}_i=(a_1,\ldots,a_n)$ where $a_1=d-1$ and $a_j=1$ if $j=i$ or $a_j=0$ otherwise, notably each $\mathbf{a}_i \in A$. We can then then re-write $(t,q_2,\ldots,q_n)$: \begin{align*}
    (t,q_2, \ldots,q_n)&=\left(t-\sum_{i=2}^n (d-1)q_i + \sum_{i=2}^n(d-1)q_i,q_2,\ldots,q_n \right)\\ 
    &= \left(t-\sum_{i=2}^n (d-1)q_i,0,\ldots,0 \right) +q_2\mathbf{a}_2+\cdots+q_n\mathbf{a}_n 
\end{align*} Furthermore, $t-\sum_{i=2}^n (d-1)q_i>0$ since $t$ was assumed to be larger than $(n-1)(d^2-d)$, and we may replace each $q_i$ with the upper bound $d$.

Finally, a direct computation shows that $t-\sum_{i=2}^n (d-1)q_i$ is a multiple of $d$. Consequently, $\left( t-\sum_{i=2}^n (d-1)q_i,0,\ldots,0 \right)\in A$, and thus $\mbe \in A$.
\end{proof}

\begin{srmk}\label{rmk: any multipinch has finite combinatorial cokernel}
In fact, the proof above shows that if $B$ is a semigroup generated by any set $U$ with $S\subset U \subset T$, then $\mca_{n,d}\setminus B$ is also a finite set.\end{srmk}

We now fix $d>2$, $n \ge 2$, and a semigroup $B$ generated as in the remark above. The homological properties of the affine semigroup ring $k[B]$ are similar to the $\max(\mbm)<d-1$ case for single pinches, as the cokernel $k[B]\rightarrow V_{n,d}$ is still finite-dimensional over $k$ and hence finite length as a $k[B]$-module.

\begin{sthm}\label{thm: multipinch is not CM}
Let $B$ be as in the paragraph above. Then, $\depth(k[B]) = 1$ and so $k[B]$ is not Cohen-Macaulay.
\end{sthm}

\begin{proof}
As in the $d>2$, $\max(\mbm)<d-1$ case of \Cref{lem: depth of pinched veroneses}, $\mca_{n,d}\setminus B$ is finite so the cokernel $C$ of $k[B]\rightarrow V_{n,d}$ is dimension zero as a $k[B]$-module. So, letting $\mfm$ be the homogeneous maximal ideal of $k[B]$, we get $0 \neq H^0_\mfm(C)=C\simeq H^1_\mfm(k[B])$, which shows that $\depth(k[B])=1$. Since $k[B]$ still contains the system of parameters $x_1^d,\ldots,x_n^d$ of $V_{n,d}$, we also know $\dim(k[B])=n$ which shows that $k[B]$ is not Cohen-Macaulay.  
\end{proof}

We now show these rings $k[B]$ are $F$-nilpotent.

\begin{sthm}\label{thm: pinching all but rook and knight moves is F-nilpotent}
Let $B$ be as established above. Then, the affine semigroup ring $k[B]$ is $F$-nilpotent.
\end{sthm}

\begin{proof}
Write $R=k[B]$, $V=V_{n,d}=k[\mca_{n,d}]$, and $\mfm$ and $\mfn$ for the homogeneous maximal ideals of $R$ and $V$ respectively. Then, $\sqrt{\mfm V}= \mfn$ since $x_i^d \in R$ for each $i$. Furthermore, the cokernel of $R\rightarrow V$ is concentrated in finitely many strictly positive degrees since $\mca_{n,d}\setminus B \subset \mca_{n,d}\setminus A$, and is hence nilpotent under the induced Frobenius action. Thus we can apply \Cref{lem: nilpotent cokernel F-rational} to conclude that $R$ is $F$-nilpotent, as required.  
\end{proof}

We are now ready to compute an upper bound on the Frobenius test exponents for these multi-pinched Veronese rings.

\begin{sthm}\label{thm: fte bounds for multipinches}
Let $B$ be a semigroup as defined in \Cref{thm: pinching all but rook and knight moves is F-nilpotent}, with $B\subset \mca_{n,d}$ for $n>2$. Then, $\fte(k[B])\le n\lceil \log_p((n-1)(d^2-d))\rceil$.
\end{sthm}

\begin{proof}
Since $\mca_{n,d}\setminus B\subset \mca_{n,d}\setminus A$ is finite, we have the cokernel $C$ of $k[B]\rightarrow V_{n,d}$ is a finite-dimensional $k$-vector space and is thus dimension $0$ as a $k[B]$-module. We can then completely analyze the local cohomology of $k[B]$; for simplicity we let $R=k[B]$, $V=V_{n,d}$, and $\mfm$ and $\mfn$ be the homogeneous maximal ideals of $R$ and $V$ respectively. Then, $\sqrt{\mfm V} = \mfn$, and we get that $H^0_\mfm(C)=C \simeq H^1_\mfm(R)$, $H^j_\mfm(R)=H^j_\mfm(V) = 0$ for $1<j< n$, and $H^n_\mfm(R)\simeq H^n_\mfm(V)$.

Then, similarly to the proof of \Cref{thm: hsl numbers for pinched veronese in 2 vars}, $\hsl(H^j_\mfm(R)) = 0$ for $j \neq 1$ and $\hsl(H^1_\mfm(R))$ is the minimum $e \in \mbn$ so that $\overline{F}^e(C)=0$, which we can see is bounded above by $\lceil\log_p((n-1)(d^2-d))\rceil$ by \Cref{lem: combinatorial cokernel multipinch case}. Thus, in Quy's upper bound for Frobenius test exponents, we get $\fte(R)\le \binom{n}{1} \hsl(H^1_\mfm(R)) \le n\lceil\log_p((n-1)(d^2-d))\rceil$. 
\end{proof}

This provides a coarse upper bound independent of which particular semigroup $B$ is chosen. If $|\mbm|=d$ and $\max(\mbm)<d-1$, then $\mca_{2,d,\mbm}$ is a semigroup for which \Cref{thm: fte bounds for multipinches} applies, but we have shown that $\fte(\mcp_{2,d,\mathbf{m}})\le 2$, which is much sharper than the bound of $2 \lceil \log_p(d^2-d)\rceil$ provided above.

We conclude the paper with some further questions.

\begin{question}
We have shown in \Cref{thm: multipinch is not CM} that pinching any number of generators $x_1^{m_1}\cdots x_n^{m_n}$ of $V_{n,d}$ where $\max(m_1,\ldots,m_n)<d-1$ provides a non Cohen-Macaulay affine semigroup ring. Can we determine all subsets $S$ of $T_{n,d}$ which generate semigroups $A$ such that $k[A]$ is Cohen-Macaulay?
\end{question}

\begin{question}
We have shown in \Cref{thm: omitted example is F-pure} that the ring $\mcp_{3,2,(1,1,0)}$ is $F$-pure when $p>2$, but it seems likely that $\mcp_{n,2,\mbm}$ for $\max(\mbm)=1$ is also $F$-pure when $p>2$. However, since $\mcp_{n,2,\mbm}$ is not even Cohen-Macaulay for $n>3$, the same technique we used to show $\mcp_{3,2,\mbm}$ is $F$-pure will not apply. Can we show that $\mcp_{n,2,\mbm}$ is always $F$-pure when $p>2$ and $\max(\mbm)=1$?   
\end{question}

\begin{question}
Hochster's theorem establishes normality as a sufficient intrinsic condition on an affine semigroup $A$ so that $k[A]$ is Cohen-Macaulay, and in prime characteristic $p>0$, normality further implies that $k[A]$ is $F$-regular, a restrictive $F$-singularity type. 

In the $\max(\mbm)<d$ case, the proofs we provide of $F$-nilpotence and $F$-injectivity depend on extrinsic qualities, i.e. the embedding $\mca_{n,d,\mbm}\subset \mca_{n,d}$. Can we determine intrinsic conditions on an affine semigroup so its corresponding affine semigroup ring has a certain classes of $F$-singularity?
\end{question}

\bibliographystyle{alpha}
\bibliography{References}
\end{document}